\documentclass[12pt]{amsart}
\usepackage{amsmath,amsthm,amsfonts,amscd,amssymb,eucal,latexsym,mathrsfs, appendix, yhmath, setspace}
\usepackage[numbers,sort&compress]{natbib}
\usepackage[all,cmtip]{xy}
\usepackage{enumerate}
\usepackage{color}

\newtheorem{theorem}{Theorem}[section]
\newtheorem{corollary}[theorem]{Corollary}
\newtheorem{lemma}[theorem]{Lemma}
\newtheorem{proposition}[theorem]{Proposition}

\theoremstyle{definition}
\newtheorem{definition}[theorem]{Definition}

\newtheorem{notation}[theorem]{Notation}
\newtheorem{example}[theorem]{Example}

\newcommand{\diag}{{\rm diag}}

\newcommand{\cI}{{\mathcal I}}

\newcommand{\cM}{{\mathcal M}}

\newcommand{\Fb}{{\mathbb F}}

\newcommand{\Nb}{{\mathbb N}}

\newcommand{\Rb}{{\mathbb R}}

\newcommand{\Zb}{{\mathbb Z}}

\newcommand{\supp}{{\rm supp}}

\newcommand{\sgn}{{\rm sgn}}

\allowdisplaybreaks

\begin{document}

\title{Bernoullicity of Lopsided Principal Algebraic Actions}

\author{Hanfeng Li}
\address{Hanfeng Li,
Department of Mathematics, SUNY at Buffalo, Buffalo, NY 14260-2900, USA
}
\email{hfli@math.buffalo.edu}

\author{Kairan Liu}
\address{Kairan Liu,
College of Mathematics and Statistics, Chongqing University, Chongqing 401331, P.R. China}
\email{lkr111@cqu.edu.cn}

\date{July 31, 2022}


\begin{abstract}
We show that the principal algebraic actions of countably infinite groups associated to lopsided elements in the integral group ring satisfying some orderability condition are Bernoulli.
\end{abstract}

\maketitle

\section{Introduction} \label{S-intro}

Actions of countably infinite groups $\Gamma$ on compact metrizable abelian groups $X$ via continuous automorphisms attracted much attention since the beginning of ergodic theory. There is a natural one-to-one correspondence between such actions and the countable left modules over the integral group ring $\Zb\Gamma$ of $\Gamma$, whence the name {\it algebraic actions} for such actions. The algebraic actions automatically preserve the normalized Haar measure $\mu_X$ of $X$ \cite{Halmos43}, thus can be studied as a probability-measure-preserving action $\Gamma\curvearrowright (X, \mu_X)$. In this article we are concerned with the Bernoullicity of $\Gamma\curvearrowright (X, \mu_X)$ for algebraic actions.

For algebraic actions $\Gamma\curvearrowright X$ of $\Zb$, the ergodicity, CPE (completely positive entropy), and Bernoullicity of $\Gamma\curvearrowright (X, \mu_X)$ are all equivalent, as shown in a series of papers culminating in \cite{Lind77, MT78}.

For algebraic actions of $\Zb^d$ with $d>1$, ergodicity and CPE are no longer equivalent, as any ergodic algebraic action of $\Zb$ can be treated as an ergodic algebraic action of $\Zb^d$ with zero entropy via composing it with the projection $\Zb^d\rightarrow \Zb$ to the first coordinate. Rudolph and Schmidt showed that CPE and Bernoullicity are still equivalent for algebraic actions of $\Zb^d$ \cite{RS95}.

Not much is known about Bernoullicity of algebraic actions of general countably infinite amenable groups, even though one has the Ornstein-Weiss theory for Bernoulli actions of such groups \cite{OW87}. For instance, it is unknown whether essentially free CPE and Bernoullicity are equivalent for algebraic actions of such groups.

For algebraic actions of general countably infinite (possibly non-amenable) groups, very little is known.
Despite that much  progress on Bernoulli actions was made in the last decade such as the entropy theory for
actions of sofic groups \cite{Bowen10}, the extension of Sinai's theorem about Bernoulli factors \cite{Seward20},
and the isomorphism of Bernoulli actions with equal base entropy  \cite{Bowen12, Seward22}, some of the key results
the Ornstein-Weiss theory do not hold anymore \cite{Bowen20}: for example, a result of Popa says that for any
countably infinite  group $\Gamma$ with property (T) and any infinite compact metrizable abelian group $K$,
the algebraic action $\Gamma\curvearrowright K^\Gamma/K$ (the quotient group of $K^\Gamma$ by the closed
subgroup of constant points) is not Bernoulli \cite{Popa06, PS07}, in particular, essentially free factors of
Bernoulli actions may fail to be Bernoulli. On the affirmative side, Ornstein and Weiss observed that the
algebraic action $\Fb_2\curvearrowright (\Zb/2\Zb)^{\Fb_2}/(\Zb/2\Zb)$ of the free group $\Fb_2$ with $2$
generators is Bernoulli \cite{OW87}. This was extended to the algebraic action $\Gamma\curvearrowright K^\Gamma/K$ for any free product $\Gamma$ of finitely many countably infinite amenable groups and any nontrivial compact metrizable abelian group $K$ by Meesschaert, Raum, and Vaes \cite{MRV13}.

Recently Lind and Schmidt established the Bernoullicity for an interesting class of algebraic actions \cite{LS22}.
For any $f$ in the integral group ring $\Zb\Gamma$ (Section~\ref{S-group algebra}) one has the induced algebraic
 $\Gamma$-action on the Pontryagin dual $X_f$ of the left $\Zb\Gamma$-module $\Zb\Gamma/(\Zb\Gamma f)$, called
 the {\it principal algebraic action} associated to $f$ (Section~\ref{S-algebraic}). When $f\in \Zb\Gamma$ is
  invertible in $\ell^1_\Rb(\Gamma)$ (equivalently when $\Gamma\curvearrowright X_f$ is expansive \cite{DS07}),
  for the finite set $S=\{0, \dots, \|f\|_1-1\}$ one has a natural continuous surjective $\Gamma$-equivariant
  map $\phi_f: S^\Gamma\rightarrow X_f$ (see \eqref{E-map}), which can be thought of as a symbolic cover of
  $\Gamma\curvearrowright X_f$. When $\Gamma$ is equipped with a group homomorphism
   $[\cdot]: \Gamma\rightarrow \Zb$, $a, b$ are distinct elements of $\Gamma$ with $[a]=[b]=1$
   (for example, when $\Gamma=\Fb_2$ with the $2$ generators $a, b$), and $f=M-a-b\in \Zb\Gamma$ for some
   integer $M\ge 3$, Lind and Schmidt showed that the restriction of $\phi_f$ to $\{0, \dots, M-1\}^\Gamma$
   sends the product measure $\nu_M^\Gamma$ to $\mu_{X_f}$ and is injective on a conull set, where
   $\nu_M$ is the uniform probability measure on $\{0, \dots, M-1\}$, thus provides an isomorphism
   between $\Gamma\curvearrowright (\{0, \dots, M-1\}^\Gamma, \nu^\Gamma)$ and
   $\Gamma\curvearrowright (X_f, \mu_{X_f})$ \cite[Theorem 7.1]{LS22}. A nice feature
    of this isomorphism is that it is explicit and continuous on the shift space
    $\{0, \dots, M-1\}^\Gamma$. Lind and Schmidt conjecture that their result holds more generally
    when $f=M-\sum_{s\in I}f_s s\in \Zb\Gamma$, where $I\subseteq \Gamma$ is finite and $[s]\ge 1$ and $f_s>0$ for every $s\in I$ and $M>\sum_{s\in I}f_s$ \cite[Conjecture 8.1]{LS22}.

Hayes extended the factor part of the Lind-Schmidt result to a more general situation \cite{Hayes22}.
The element $f=\sum_{s\in \Gamma} f_s s\in \Zb\Gamma$ is called {\it lopsided} if there is some $s_0\in \Gamma$ such that $|f_{s_0}|>\sum_{s\in \Gamma\setminus \{s_0\}}|f_s|$.
We say a lopsided $f$ is {\it positively lopsided} if furthermore there is some right-invariant partial order
$\le$ on $\Gamma$ such that $s_0<t$ for every $t\in \Gamma\setminus \{s_0\}$ with $f_t\neq 0$ (in the presence
of a homomorphism $[\cdot]: \Gamma\rightarrow \Zb$ as in \cite{LS22}, one can use the right-invariant partial
order given  by $s<t$ when $[s]<[t]$). When $f\in \Zb\Gamma$ is positively lopsided, Hayes showed that the
restriction of $\phi_f$ to $\{0, \dots, |f_{s_0}|-1\}^\Gamma$ sends the product measure $\nu_{|f_{s_0}|}^\Gamma$
to $\mu_{X_f}$, thus $\Gamma\curvearrowright (X_f, \mu_{X_f})$ is a factor of a Bernoulli action
 \cite[Corollary 5.2]{Hayes22}, and  conjectures that $\phi_f$ is an isomorphism between
$\Gamma\curvearrowright (\{0, \dots, |f_{s_0}|-1\}^\Gamma, \nu_{|f_{s_0}|}^\Gamma)$ and
$\Gamma\curvearrowright (X_f, \mu_{X_f})$ \cite[Conjecture 1]{Hayes22}.
If furthermore $\Gamma$ is amenable and torsion-free, then using \cite{OW87}
he concludes that $\Gamma\curvearrowright (X_f, \mu_{X_f})$ is Bernoulli \cite[Theorem 1.2]{Hayes22},
though it's still not clear whether the restriction of $\phi_f$ to $\{0, \dots, |f_{s_0}|-1\}^\Gamma$
is an isomorphism of measure spaces.

In this work we consider not only algebraic actions associated to elements in $\Zb\Gamma$, but also algebraic actions associated to square matrices over $\Zb\Gamma$, as some new phenomenon shows up. For any $n\in \Nb$ and $f\in M_n(\Zb\Gamma)$, we have the generalized principal algebraic action $\Gamma\curvearrowright X_f$ (see \eqref{E-explicit}). When $f\in M_n(\Zb\Gamma)$ is invertible in $M_n(\ell^1_{\Rb}(\Gamma))$, one still has the continuous $\Gamma$-equivariant map $\phi_f: S^\Gamma\rightarrow X_f$ for any nonempty finite subset $S$ of $\Zb^n$ (see \eqref{E-map}). It turns out that there are two ways to extend lopsidedness to matrices: row lopsidedness and column lopsidedness (Definitions~\ref{D-lopsided} and \ref{D-positive lopsided}). 
When $f\in M_n(\Zb\Gamma)$ is row or column lopsided, there is a  finite symbol set $S_f\subseteq \Zb^n$ (Notation~\ref{N-symbol}) which is of the form $\prod_{j=1}^nS_j$ for some $S_1, \dots, S_n\subseteq \Zb$ and plays the role of $\{0, \dots, |f_{s_0}|-1\}$ for lopsided $f\in \Zb\Gamma$.

It turns out that Hayes' result holds whenever $f\in M_n(\Zb\Gamma)$ is either positively row or column lopsided (Proposition~\ref{P-Haar}), while we only know that injectivity holds  when $f$ is positively row lopsided:

\begin{theorem} \label{T-injective}
Let $\Gamma$ be a countably infinite group. Let $f\in M_n(\Zb\Gamma)$ be positively row lopsided. Let $\nu$ be a probability measure on $S_f$ such that its marginal distribution on $S_j$ is the uniform probability measure of $S_j$ for all $j=1, \dots, n$. Then the set $\{y\in S_f^\Gamma: |\phi^{-1}_f(\phi_f(y))\cap S_f^\Gamma|>1\}$ has $\nu^\Gamma$ measure $0$. Thus $\phi_f$ is a conjugation between $\Gamma\curvearrowright (S_f^\Gamma, \nu^\Gamma)$  and $\Gamma\curvearrowright (X_f, (\phi_f)_*\nu^\Gamma)$.
\end{theorem}

We do not know whether Theorem \ref{T-injective} still holds when $f\in M_n(\Zb\Gamma)$ is positively column lopsided.

Given continuous actions of $\Gamma$ on compact metrizable spaces $Y$ and $Z$, and $\Gamma$-invariant Borel probability measures $\mu$ and $\mu'$ on $Y$ and $Z$ respectively, one says that $\Gamma\curvearrowright (Y, \mu)$ and $\Gamma\curvearrowright (Z, \mu')$ are {\it almost topologically conjugate} if there are residual $\Gamma$-invariant Borel sets $Y'\subseteq Y$ and $Z'\subseteq Z$ with $\mu(Y')=\mu'(Z')=1$ and a bimeasurable, bicontinuous $\Gamma$-equivariant isomorphism $(Y', \mu)\rightarrow (Z', \mu')$.
Combining Theorem~\ref{T-injective} and the result of Hayes, we have the following corollaries.

\begin{corollary} \label{C-Bernoulli}
Let $\Gamma$ be a countably infinite group.
Let $f\in M_n(\Zb\Gamma)$ be positively  row lopsided. Let $\nu$ be the uniform probability measure on $S_f$. Then $\phi_f$ gives rise to an almost topological conjugation between  $\Gamma\curvearrowright (S_f^\Gamma, \nu^\Gamma)$  and $\Gamma\curvearrowright (X_f, \mu_{X_f})$.
\end{corollary}

The case $n=1$ of Corollary~\ref{C-Bernoulli} proves the conjectures of Lind-Schmidt and Hayes.

\begin{corollary} \label{C-almost}
Let $\Gamma$ be a countably infinite group.
Let $f\in M_n(\Zb\Gamma)$ and $h\in M_m(\Zb\Gamma)$ be positively row lopsided such that $|S_f|=|S_h|$. Then $\Gamma\curvearrowright (X_f, \mu_{X_f})$ and
$\Gamma\curvearrowright (X_h, \mu_{X_h})$ are almost topologically conjugate.
\end{corollary}

Corollary~\ref{C-Bernoulli} leaves open the question whether $\Gamma\curvearrowright (X_f, \mu_{X_f})$ is Bernoulli for any torsion-free countably infinite group $\Gamma$ and any $f\in \Zb\Gamma$ invertible in $\ell^1_\Rb(\Gamma)$.

We remark that Hayes also proved his result for some $f\in \Zb\Gamma$ with a formal $\ell^2$ inverse \cite[Corollary 5.2]{Hayes22}, and conjectures that
$\phi_f$ is also injective on a conull set in this situation \cite[Conjecture 1]{Hayes22}. Our method does not apply to such case.

This paper is organized as follows. We recall some basic facts about group algebras, algebraic actions, and right-invariant partial orders in Section~\ref{S-prelim}. The notions of row or column lopsided matrices are introduced in Section~\ref{S-lopsided}. We prove Theorem~\ref{T-injective} in Section~\ref{S-Bernoulli}. A proof of Hayes' result is included in Section~\ref{S-Haar}.

\noindent{\it Acknowledgements.}
H.L. was partially supported by NSF grant DMS-1900746. We are grateful to the referee for helpful comments.

\section{Preliminaries} \label{S-prelim}

In this section we recall some basic facts and set up some notations.
Throughout this paper, $\Gamma$ will be a countably infinite group with identity element $e_\Gamma$. For $n\in \Nb$, we write $[n]$ for $\{1, \dots, n\}$.

\subsection{Group algebras} \label{S-group algebra}

We refer the reader to \cite{Passman77} for general information about group rings.

The {\it integral group ring} $\Zb\Gamma$ of $\Gamma$ is the set of all finitely supported functions $f: \Gamma\rightarrow \Zb$. We shall write $f$ as $\sum_{s\in \Gamma}f_s s$, where $f_s\in \Zb$ for all $s\in \Gamma$ and $f_s=0$ for all except finitely many $s\in \Gamma$. The set $\{s\in \Gamma: f_s\neq 0\}$ is denoted by $\supp(f)$. The addition and multiplication of $\Zb\Gamma$ are given by
$$\sum_{s\in \Gamma}f_s s+\sum_{s\in \Gamma}g_ss=\sum_{s\in \Gamma}(f_s+g_s)s,$$
and
\begin{align} \label{E-product}
(\sum_{s\in \Gamma}f_s s\big)\big(\sum_{s\in \Gamma}g_s s\big)=
\sum_{t\in \Gamma}\big(\sum_{s\in \Gamma}f_sg_{s^{-1}t}\big)t=\sum_{t\in \Gamma}\big(\sum_{s\in \Gamma}f_{ts^{-1}}g_s\big)t.
\end{align}
There is also an involution $f\mapsto f^*$ on $\Zb\Gamma$ given by
\begin{align} \label{E-*}
(\sum_{s\in \Gamma}f_s s\big)^*=\sum_{s\in \Gamma}f_{s^{-1}}s.
\end{align}
One has $(f+g)^*=f^*+g^*$ and $(fg)^*=g^*f^*$ for all $f, g\in \Zb\Gamma$.

We also have the Banach space $\ell^\infty_\Rb(\Gamma)$ of all bounded functions $\Gamma\rightarrow \Rb$ with the canonical norm $\|\cdot\|_\infty$ and the Banach space $\ell^1_\Rb(\Gamma)$ of all absolutely summable functions $\Gamma\rightarrow \Rb$ with the canonical norm $\|\cdot \|_1$. We shall also write the elements of $\ell^\infty_\Rb(\Gamma)$ and $\ell^1_\Rb(\Gamma)$ formally as $\sum_{s\in \Gamma}f_ss $ with $f_s\in \Rb$ for all $s\in \Gamma$. Then
$$ \big\|\sum_{s\in \Gamma}f_ss\big\|_\infty=\sup_{s\in \Gamma}|f_s|$$
for $f\in \ell^\infty_\Rb(\Gamma)$, and
$$ \big\|\sum_{s\in \Gamma}f_ss\big\|_1=\sum_{s\in \Gamma}|f_s|$$
for $f\in \ell^1_\Rb(\Gamma)$. Note that $\ell^1_\Rb(\Gamma)$ is a $*$-algebra with multiplication and $*$-operations given by \eqref{E-product} and \eqref{E-*} respectively. It is a Banach $*$-algebra in the sense that $\|fg\|_1\le \|f\|_1\|g\|_1$ and $\|f^*\|_1=\|f\|_1$ for all $f, g\in \ell^1_\Rb(\Gamma)$.

For $n\in \Nb$, we shall write elements $f\in M_n(\ell^1_\Rb(\Gamma))$ as $f=(f^{(km)})_{k, m\in [n]}$. The $*$-operation extends to $M_n(\ell^1_\Rb(\Gamma))$ by $(f^*)^{(km)}=(f^{(mk)})^*$ for all $f\in M_n(\ell^1_\Rb(\Gamma))$ and $k, m\in [n]$. Then we still have $(fg)^*=g^*f^*$ for all $f, g\in M_n(\ell^1_\Rb(\Gamma))$.

On $M_n(\ell^1_\Rb(\Gamma))$ we have the norms $\|\cdot\|_{\infty, 1}$ and $\|\cdot\|_{1, \infty}$ given by
$$\|f\|_{\infty, 1}:=\max_{k\in [n]}\sum_{m\in [n]}\|f^{(km)}\|_1, \mbox{ and } \|f\|_{1, \infty}:=\max_{m\in [n]}\sum_{k\in [n]}\|f^{(km)}\|_1$$
for $f=(f^{(km)})_{k, m\in [n]}\in M_n(\ell^1_\Rb(\Gamma))$.
These two norms are equivalent, and are related via $\|f^*\|_{\infty, 1}=\|f\|_{1, \infty}$ for all $f\in M_n(\ell^1_\Rb(\Gamma))$. We have  $\|fg\|_{\infty, 1}\le \|f\|_{\infty, 1}\cdot \|g\|_{\infty, 1}$ and $\|fg\|_{1, \infty}\le \|f\|_{1, \infty}\cdot \|g\|_{1, \infty}$ for all $f, g\in M_n(\ell^1_\Rb(\Gamma))$.
Then $M_n(\ell^1_\Rb(\Gamma))$ is a unital Banach algebra under either of these two norms.

For any $f=(f^{(km)})_{k, m\in [n]}\in M_n(\Zb\Gamma)$, we put
$$\supp(f)=\{(s, k, m): k, m\in [n], s\in \supp(f^{(km)})\}\subseteq \Gamma\times [n]^2.$$

\subsection{Algebraic actions} \label{S-algebraic}

We refer the reader to \cite{Schmidt95, KL16} for general information about algebraic actions.

For any countable abelian group $\cM$, denote by $\widehat{\cM}$ the Pontryagin dual of $\cM$, i.e. the set of all group homomorphisms $\cM\rightarrow \Rb/\Zb$. Under the pointwise addition and the topology of pointwise convergence, $\widehat{\cM}$ is a compact metrizable abelian group. Up to isomorphism, every compact metrizable abelian group arises this way. We shall denote the pairing $\widehat{\cM}\times \cM\rightarrow \Rb/\Zb$ by $\left<x, a\right>=x(a)$.

Let $\cM$ be a countable left $\Zb\Gamma$-module. Then $\Gamma$ has an induced action on $\widehat{\cM}$ via continuous automorphisms determined by
$$ \left<sx, sa\right>=\left<x, a\right>$$
for all $x\in \widehat{\cM}$, $a\in \cM$, and $s\in \Gamma$. As an example, for $\cM=\Zb\Gamma$, we have $\widehat{\Zb\Gamma}=\widehat{\bigoplus_{s\in \Gamma}\Zb}=\prod_{s\in \Gamma}\Rb/\Zb=(\Rb/\Zb)^\Gamma$, and the pairing $\widehat{\Zb\Gamma}\times \Zb\Gamma\rightarrow \Rb/\Zb$ is given by
$$ \left<x, g\right>=\sum_{s\in \Gamma}x_s g_s=(xg^*)_{e_\Gamma},$$
where the product $xg^*\in (\Rb/\Zb)^\Gamma$ for $x\in (\Rb/\Zb)^\Gamma$ and $g^*\in \Zb\Gamma$ is defined using \eqref{E-product}. The induced $\Gamma$-action on $(\Rb/\Zb)^\Gamma$ is the left-shift action given by $(sx)_t=x_{s^{-1}t}$ for all $x\in (\Rb/\Zb)^\Gamma$ and $s, t\in \Gamma$.
More generally, for any $n\in \Nb$ and $\cM=(\Zb\Gamma)^n$, we have $\widehat{(\Zb\Gamma)^n}=((\Rb/\Zb)^\Gamma)^n$, and the pairing $\widehat{(\Zb\Gamma)^n}\times (\Zb\Gamma)^n\rightarrow \Rb/\Zb$ is given by
$$ \left<x, g\right>=\sum_{k\in [n], s\in \Gamma}x_{s, k} g_{s, k}=(xg^*)_{e_\Gamma},$$
where we write $x\in ((\Rb/\Zb)^\Gamma)^n$ and $g\in (\Zb\Gamma)^n$ as row vectors so that $g^*$ is a column vector.
For any left $\Zb\Gamma$-submodule $\cI$ of $(\Zb\Gamma)^n$ and $\cM=(\Zb\Gamma)^n/\cI$, the dual $\widehat{\cM}$ is the $\Gamma$-invariant closed subgroup of $\widehat{(\Zb\Gamma)^n}$ consisting of elements annihilating $\cI$ under the pairing $\widehat{(\Zb\Gamma)^n}\times (\Zb\Gamma)^n\rightarrow \Rb/\Zb$, i.e.
$$ \widehat{(\Zb\Gamma)^n/\cI}=\{x\in ((\Rb/\Zb)^\Gamma)^n: \left<x, g\right>=0_{\Rb/\Zb} \mbox{ for all } g\in \cI\}.$$

In particular, for any $f\in M_n(\Zb\Gamma)$, we have the corresponding {\it generalized principal algebraic action}  $\Gamma\curvearrowright X_f:=\widehat{(\Zb\Gamma)^n/((\Zb\Gamma)^nf)}$. Note that for any $x\in ((\Rb/\Zb)^\Gamma)^n$, one has $(xg^*)_{e_\Gamma}=\left<x, g\right>=0_{\Rb/\Zb}$ for all $g\in (\Zb\Gamma)^n f$ exactly when $xf^*=0$. Thus
\begin{align} \label{E-explicit}
 X_f=\{x\in ((\Rb/\Zb)^\Gamma)^n: xf^*=0\}.
 \end{align}

Let $f\in M_n(\Zb\Gamma)$ be invertible in $M_n(\ell^1_\Rb(\Gamma))$. Since $M_n(\ell^1_{\Rb}(\Gamma))$ is a $*$-algebra, $f^*$ is also invertible in $M_n(\ell^1_\Rb(\Gamma))$ with $(f^*)^{-1}=(f^{-1})^*$. Denote by $\pi$ the quotient map $(\Rb^\Gamma)^n=\Rb^{\Gamma \times [n]}\rightarrow \Rb^{\Gamma\times [n]}/\Zb^{\Gamma \times [n]}=(\Rb/\Zb)^{\Gamma\times [n]}=((\Rb/\Zb)^\Gamma)^n$. For any $y\in (\Zb^n)^\Gamma\cap (\ell^\infty_\Rb(\Gamma))^n$, we have $y(f^*)^{-1}\in (\ell^\infty_\Rb(\Gamma))^n$, and from \eqref{E-explicit} one sees easily that $\pi(y(f^*)^{-1})\in X_f$. This defines a map $\phi_f: (\Zb^n)^\Gamma\cap (\ell^\infty_\Rb(\Gamma))^n\rightarrow X_f$ by
\begin{align} \label{E-map}
 \phi_f(y)=\pi(y(f^*)^{-1}).
 \end{align}
This is called the {\it homoclinic map} since $\phi_f((\Zb\Gamma)^n)$ is exactly the group of homoclinic points of $X_f$, i.e. 
elements $x$ of $X_f$ satisfying $sx\to 0_{X_f}$ as $\Gamma\ni s\to \infty$. 
For any nonempty finite set $S\subseteq \Zb^n$, the restriction of $\phi_f$ to $S^\Gamma$ is a continuous $\Gamma$-equivariant map $S^\Gamma\rightarrow X_f$.

\subsection{Right-invariant partial order} \label{S-order}

A (partial) order $\le$ on $\Gamma$ is called {\it right-invariant} if for any $s, t, \gamma\in \Gamma$ one has $s\le t$ if and only if $s\gamma \le t\gamma$. We refer the reader to \cite{Glass99, KM96, MR77} for general information about groups equipped with right-invariant (partial) orders.

Given a right-invariant partial order $\le$ on $\Gamma$, the set $P_{\le }:=\{s\in \Gamma: e_\Gamma< s\}$ of positive elements is a semigroup contained in  $\Gamma\setminus \{e_\Gamma\}$. Conversely, given any semigroup $P$ contained in $\Gamma\setminus \{e_\Gamma\}$, we have the right-invariant partial order $\le_P$ on $\Gamma$ defined by $s\le_P t$ if and only if $ts^{-1}\in P\cup \{e_\Gamma\}$. It is easily checked that this gives us a $1$-$1$ correspondence between right-invariant partial orders on $\Gamma$ and semigroups contained in $\Gamma\setminus \{e_\Gamma\}$. A right-invariant partial order $\le$ is an order on $\Gamma$ exactly when $\Gamma=P_{\le}\cup \{e_\Gamma\}\cup P_{\le}^{-1}$.

\section{Lopsided matrices} \label{S-lopsided}

\begin{definition} \label{D-lopsided}
Let $n\in \Nb$. We say $f\in M_n(\Zb\Gamma)$ is {\it row lopsided} if $f$ is of the form $M-g$ with $M, g\in M_n(\Zb\Gamma)$ such that $\supp(M)\cap \supp(g)=\emptyset$, $M=\diag(M_1s_1, \dots, M_ns_n)$ for some $s_1, \dots, s_n\in \Gamma$, $M_1, \dots, M_n\in \Zb$, and $g=(g^{(km)})_{k, m\in [n]}$ satisfy
\begin{align} \label{E-row}
|M_k|>\sum_{m\in [n]}\|g^{(km)}\|_1
\end{align}
for each $k\in [n]$. Similarly, we say $f$ is {\it column lopsided} if
instead of \eqref{E-row} we have
\begin{align*}
|M_k|>\sum_{m\in [n]}\|g^{(mk)}\|_1
\end{align*}
for each $k\in [n]$.
\end{definition}

\begin{notation} \label{N-symbol}
Let $f\in M_n(\Zb\Gamma)$ be row or column lopsided. Using the notation in Definition~\ref{D-lopsided} we put
$$S_f:=\prod_{k\in [n]}\{0, \dots, |M_k|-1\}\subseteq \Zb^n.$$
\end{notation}

\begin{lemma} \label{L-invertible}
If $f\in M_n(\Zb\Gamma)$ is either row lopsided or column lopsided, then it is invertible in $M_n(\ell^1_\Rb(\Gamma))$.
\end{lemma}
\begin{proof} Let $f\in M_n(\Zb\Gamma)$ be column lopsided.
We shall use the notation in Definition~\ref{D-lopsided}.
We have $f=M-g=(I_n-gM^{-1})M$ with $\|gM^{-1}\|_{1, \infty}\le \max_{k\in [n]}\frac{|M_k|-1}{|M_k|}<1$. Thus $f$ is invertible in $M_n(\ell^1_\Rb(\Gamma))$ with
\begin{align*}
f^{-1}=M^{-1}(I_n-gM^{-1})^{-1}=\sum_{l=0}^\infty M^{-1}(gM^{-1})^l.
\end{align*}

Similarly, if $f\in M_n(\Zb\Gamma)$ is row lopsided, then using $\|\cdot \|_{\infty, 1}$ one can show that $f$ is invertible in $M_n(\ell^1_\Rb(\Gamma))$. Another way to see this is that if $f$ is row lopsided, then $f^*$ is column lopsided, so $f^*$ is invertible in $M_n(\ell^1_\Rb(\Gamma))$,
whence $f$ is invertible in $M_n(\ell^1_\Rb(\Gamma))$ with $f^{-1}=((f^*)^{-1})^*$.
\end{proof}

\begin{definition} \label{D-positive lopsided}
We say a row (resp. column) lopsided $f\in M_n(\Zb\Gamma)$ is {\it positively row (resp. column) lopsided} if, in the notation of Definition~\ref{D-lopsided}, there is a right-invariant partial order $\le$ on $\Gamma$ such that $s_m<t$ for all $t\in \bigcup_{k\in [n]}\supp(g^{(km)})$ and $m\in [n]$, equivalently,
there is a semigroup $P$ contained in  $\Gamma\setminus \{e_\Gamma\}$ such that  $P\supseteq \bigcup_{k, m\in [n]}\supp(g^{(km)}s_m^{-1})$.
\end{definition}

\begin{example} \label{E-lopsided}
Let $a, b\in \Gamma$ such that the subsemigroup of $\Gamma$ generated by $a$ and $b$ does not contain $e_\Gamma$.
Put
\begin{align*}
f=\left[\begin{matrix} 7-2a^3b^2 & 3a+b^7a\\ 5b^2a^4 & -10-3a^8 \end{matrix}\right], \quad \quad h=\left[\begin{matrix} 8a^{-1}-2a^3b^2a^{-1} & 3ab+b^7ab\\ 5b^2a^3 & -10b-3a^8b \end{matrix}\right]
\end{align*}
in $M_2(\Zb\Gamma)$. Then $f$ is positively row lopsided but not column lopsided, while $h$ is both positively row and column lopsided. We have
$S_f=\{0, \dots, 6\}\times \{0, \dots, 9\}\subseteq \Zb^2$ and $S_h=\{0, \dots, 7\}\times \{0, \dots, 9\}\subseteq \Zb^2$.
\end{example}

\section{Bernoullicity} \label{S-Bernoulli}

In this section we prove Theorem~\ref{T-injective}. Before moving on to the more complicated situation, it might be helpful for the reader's motivation to have a brief outline of the
proof for the following example.

\begin{example} Let $f=M-f_aa-f_bb\in\mathbb{Z}\Gamma$ 
such that 
$a$ and $b$ are distinct elements in some semigroup $P$ contained in $\Gamma\setminus \{e_{\Gamma}\}$,
and  $0<f_a\leq f_b$, $f_a+f_b<M$.
  Let $\nu$ be the uniform probability measure on $S_f=\{0, \dots, M-1\}$. For any subset $E$ of $\Gamma$, denote by $\nu^E$ the product measure on $S_f^E$ with base measure $\nu$, and by $\pi_E$ the restriction map $S_f^\Gamma\rightarrow S_f^E$.
  We fix an integer $N\geq M\Vert(f^*)^{-1}\Vert_{1}$, set $V=\{-N,-N+1,\cdots,N-1,N\}^{\Gamma}\backslash\{0\}$ and
  $$Z=\{(y,c)\in S_f^{\Gamma}\times V\colon y+cf^{*}\in S_f^{\Gamma}\}.$$
  Denote by $\varphi$ the projection $S_f^{\Gamma}\times V\to S_f^{\Gamma}$.
  It is easily checked that $\varphi(Z)$ is exactly the set of $y\in S_f^\Gamma$ with $|\phi_f^{-1}(\phi_f(y))|>1$ (see Lemma~\ref{L-bad set}).
  Then Theorem~\ref{T-injective} for this $f$ amounts to $\nu^{\Gamma}(\varphi(Z))=0$, which we prove now. 

For every $(j,s)\in[N]\times\Gamma$, we put
  $$ Z^+_{j, s}=\{(y, c)\in Z: \|c\|_\infty=c_{s}=j\} \mbox{ and } Z^-_{j, s}=\{(y, c)\in Z: \|c\|_\infty=-c_{s}=j\}.$$
Then $\varphi(Z^{\dag}_{j, s})$ is a closed subset of $S_f^\Gamma$ for every $(j,s)\in[N]\times \Gamma$ and $\dag\in \{+, -\}$,
and
$$Z=\bigcup_{(j, s)\in [N]\times \Gamma, \dag\in \{+, -\}}Z^\dag_{j, s}.
$$
Denote by $\overline{M-1}$ the element in $S_f^\Gamma$ taking value $M-1$ at every $s\in \Gamma$. 
We note that for any $(y, c)\in S_f^\Gamma\times V$ and  $(j,s)\in[N]\times\Gamma$, one has $(y,c)\in Z_{j,s}^+$ if and only if $(\overline{M-1}-y,-c)\in Z_{j,s}^-$. 
Thus we only need to show that
$\nu^{\Gamma}\big(\varphi\big(\bigcup_{(j,s)\in[N]\times\Gamma}Z_{j,s}^+\big)\big)=0$.

For every $(j,s,i)\in [N]\times \Gamma\times S_f$ and $A\subseteq \{a,b\}$, we set
$$Z^+_{j,s}(A,i)=\big\{(y,c)\in Z^+_{j,s}: c_{st}=j\  \text{for}\ t\in A, c_{st}\neq j\  \text{for}\  t\in\{a,b\}\backslash A\  \text{and}\ y_{s}=i \big\}.$$
It is easily checked that $Z^+_{j,s}(\emptyset, i)=\emptyset$ for every $i\in S_f$, whence
$$Z^+_{j,s}=\bigcup_{i=0}^{M-1}H_{j,s}(i),$$
where $H_{j,s}(i):=Z^+_{j,s}(\{a\},i)\cup Z^+_{j,s}(\{b\},i)\cup Z^+_{j,s}(\{a,b\},i)$.
For every $(j,s)\in [N]\times \Gamma$, we note the following three facts.

\textbf{Fact 1: }
\begin{align}\label{fact 1}
  H_{j,s}(i)&=\begin{cases}
    Z^+_{j,s}(\{a\},i)\cup Z^+_{j,s}(\{b\},i)\cup Z^+_{j,s}(\{a,b\},i), & \mbox{if }  0 \leq i\leq f_a-1;\\
    Z^+_{j,s}(\{b\},i)\cup Z^+_{j,s}(\{a,b\},i), & \mbox{if }  f_a\leq i\leq f_b-1;\\
    Z^+_{j,s}(\{a,b\},i), & \mbox{if } f_b\leq i\leq f_a+f_b-1;\\
    \emptyset, & \mbox{if }  f_a+f_b\leq i\leq M-1.
  \end{cases}
\end{align}
In fact, for every $A\subseteq \{a,b\}$, if there exists $(y,c)\in Z_{j,s}^+(A,i)$, one has $(y+cf^*)_s\in S_f$, then
\begin{align*}
  M-1\geq y_s+Mc_s-f_ac_{sa}-f_bc_{sb}&=i+Mj-\big(\sum_{t\in A}f_tj+\sum_{t\in\{a,b\}\backslash A}f_tc_{st}\big)\\
  &\geq i+Mj-\big(\sum_{t\in A}f_tj+\sum_{t\in\{a,b\}\backslash A}f_t(j-1)\big).
\end{align*}
Thus if $Z_{j,s}^+(A,i)$ is not empty, then $i\leq \sum_{t\in A}f_t-1-(M-f_a-f_b)(j-1)\leq\sum_{t\in A}f_{t}-1$ and (\ref{fact 1}) follows.

\textbf{Fact 2:} Put $\bar{P}=P\cup \{e_\Gamma\}$. For every $C\subseteq\{a,b\}$, denote by $Y_{j,s,C}$ the set of $y\in S_f^{sP}$ satisfying $y\in  \pi_{st\bar{P}}\varphi(Z^{+}_{j, st})\times S_f^{sP\setminus st\bar{P}}$ for every $t\in C$ and $y\not\in  \pi_{sa\bar{P}}\varphi(Z^{+}_{j, st})\times S_f^{sP\setminus sa\bar{P}}$ for every $t\in \{a,b\}\backslash C$.
It is clear that the family $\{Y_{j,s, C}: C\subseteq \{a,b\}\}$ is a finite Borel partition of $S_f^{sP}$, and for every $A\subseteq \{a,b\}$ and $i\in S_f$ one has 
$$\pi_{sP}\varphi(Z^{+}_{j,s}(A,i))\subseteq\bigcup_{A\subseteq C\subseteq \{a,b\}}Y_{j,s,C}.$$

\textbf{Fact 3:} For each $t\in\{a,b\}$, one has $\nu^{sP}(Y_{j, s,\{t\}}\cup Y_{j,s,\{a,b\}})\leq \nu^{st\bar{P}}(\pi_{st\bar{P}}\varphi(Z^+_{j,st}))$, since $Y_{j,s,\{t\}}\cup Y_{j,s,\{a,b\}}\subseteq S_f^{sP\setminus st\bar{P}}\times \pi_{st\bar{P}}\varphi(Z^+_{j,st})$.

\bigskip
By Facts 1--3, for every $(j,s)\in[N]\times\Gamma$ we have
\begin{align}\label{Z^+}
  \nu^{s\bar{P}}(\pi_{s\bar{P}}\varphi(Z^+_{j,s}))&=\nu^{s\bar{P}}\bigg(\pi_{s\bar{P}}\varphi\big(\bigcup_{i=0}^{M-1}H_{j,s}(i)\big)\bigg)\leq \sum_{i=0}^{M-1}\nu^{s\bar{P}}(\pi_{s\bar{P}}\varphi(H_{j,s}(i)))\nonumber\\
  &= \sum_{i=0}^{M-1}\nu(\{i\})\cdot\nu^{sP}(\pi_{sP}\varphi(H_{j,s}(i)))\nonumber\\
  &\overset{\text{Fact 1}}{=}\frac{1}{M}\sum_{i=0}^{f_a-1}\nu^{sP}\bigg(\pi_{sP}\varphi\big(Z^+_{j,s}(\{a\},i)\cup Z^+_{j,s}(\{b\},i)\cup Z^+_{j,s}(\{a,b\},i)\big)\bigg)\\
  &+\frac{1}{M}\sum_{i=f_a}^{f_b-1}\nu^{sP}\bigg(\pi_{sP}\varphi(Z^+_{j,s}(\{b\},i)\cup Z^+_{j,s}(\{a,b\},i))\bigg)\nonumber\\
  &+\frac{1}{M}\sum_{i=f_b}^{f_a+f_b-1}\nu^{sP}\bigg(\pi_{sP}\varphi(Z^+_{j,s}(\{a,b\},i))\bigg)\nonumber\\
  &\overset{\text{Fact 2}}{\leq}\frac{1}{M}\sum_{i=0}^{f_a-1}\nu^{sP}(Y_{j,s,\{a\}}\cup Y_{j,s,\{b\}}\cup Y_{j,s,\{a,b\}})\nonumber\\
  &+\frac{1}{M}\sum_{i=f_a}^{f_b-1}\nu^{sP}(Y_{j,s,\{b\}}\cup Y_{j,s,\{a,b\}})+\frac{1}{M}\sum_{i=f_b}^{f_a+f_b-1}\nu^{sP}(Y_{j,s,\{a,b\}})\nonumber\\
  &=\frac{1}{M}f_a\cdot\nu^{sP}(Y_{j,s,\{a\}}\cup Y_{j,s,\{a,b\}})+\frac{1}{M}f_b\cdot\nu^{sP}(Y_{j,s,\{b\}}\cup Y_{j,s,\{a,b\}})\nonumber\\
  &\overset{\text{Fact 3}}{\leq} \frac{1}{M}f_a\cdot\nu^{sa\bar{P}}(\pi_{sa\bar{P}}\varphi(Z^+_{j,sa}))+\frac{1}{M}f_b\cdot\nu^{sb\bar{P}}(\pi_{sb\bar{P}}\varphi(Z^+_{j,sb}))\nonumber.
\end{align}
For every $j\in[N]$, if we set $p_j=\sup_{s\in\Gamma}\nu^{s\bar{P}}(\pi_{s\bar{P}}\varphi(Z^+_{j,s}))\geq 0$, by (\ref{Z^+}) we have $p_j\leq \frac{f_a+f_b}{M}p_j\le \frac{M-1}{M}p_j$, thus $p_j=0$. Therefore $\nu^{s\bar{P}}(\pi_{s\bar{P}}\varphi(Z^+_{j,s}))=0$ for every $(j,s)\in[N]\times\Gamma$.
This implies $\nu^{\Gamma}(\varphi(Z^+_{j, s}))=0$ for every $(j,s)\in[N]\times\Gamma$.
\end{example}

Now we consider the general situation, where we have to handle the complications caused by the different entries of $f$ and different signs of the coefficients of the entries of $f$. 
We start with the following lemma, which will allow us to reduce the study of $\Gamma \curvearrowright X_f$ for row or column lopsided $f$ to another matrix of better shape.

\begin{lemma} \label{L-reduction}
Let $f\in M_n(\Zb\Gamma)$ be invertible in $M_n(\ell^1_\Rb(\Gamma))$ and let $u\in M_n(\Zb\Gamma)$ be invertible in $M_n(\Zb\Gamma)$.
Let $S$ be a nonempty finite subset of $\Zb^n$ and put $Y=S^\Gamma$.
Then there is a $\Gamma$-equivariant isomorphism $\overline{\Phi}: X_f\rightarrow X_{fu^{-1}}$ of compact abelian groups such that the diagram
\begin{align} \label{E-reduction}
\xymatrix
{
 & Y \ar[dl]_{\phi_f}  \ar[dr]^{\phi_{fu^{-1}}}  \\
 X_f \ar[rr]^{\overline{\Phi}} && X_{fu^{-1}}
}
\end{align}
commutes.
\end{lemma}
\begin{proof} We have a left $\Zb\Gamma$-module isomorphism $\Phi: (\Zb\Gamma)^n\rightarrow (\Zb\Gamma)^n$  sending $a$ to $au$. Note that $\Phi((\Zb\Gamma)^n fu^{-1})=(\Zb\Gamma)^n f$. Thus $\Phi$ induces a left $\Zb\Gamma$-module isomorphism $\Phi': (\Zb\Gamma)^n/(\Zb\Gamma)^n fu^{-1}\rightarrow (\Zb\Gamma)^n/(\Zb\Gamma)^n f$ sending $a+(\Zb\Gamma)^nfu^{-1}$ to $au+(\Zb\Gamma)^nf$. At the dual level, it induces an isomorphism $\overline{\Phi}: X_f\rightarrow X_{fu^{-1}}$ of compact abelian groups sending $x$ to $xu^*$, commuting with the $\Gamma$-action.
Now clearly the diagram \eqref{E-reduction} commutes.
\end{proof}

Let $f\in M_n(\Zb\Gamma)$ be positively row lopsided. We use the notation in Definitions~\ref{D-lopsided} and \ref{D-positive lopsided}.
Put $u=\diag(\sgn(M_1)s_1, \dots, \sgn(M_n)s_n)\in M_n(\Zb\Gamma)$.
Note that $u$ is invertible in $M_n(\Zb\Gamma)$ with $u^{-1}=u^*$. Then $fu^{-1}=\diag(|M_1|, \dots, |M_n|)-gu^{-1}$ with $|M_k|>\sum_{m\in [n]}\|(gu^{-1})^{(km)}\|_1$ for all $k\in [n]$, and $P\supseteq \bigcup_{k, m\in [n]}\supp((gu^{-1})^{(km)})$. Thus $fu^{-1}$ is positively row lopsided,
and $S_f=S_{fu^{-1}}$.
By Lemma~\ref{L-reduction} we know that $\phi_f^{-1}(\phi_f(y))\cap S_f^\Gamma=\phi_{fu^{-1}}^{-1}(\phi_{fu^{-1}}(y))\cap S_{fu^{-1}}^\Gamma$ for all $y\in S_f^\Gamma=S_{fu^{-1}}^\Gamma$.
Thus Theorem~\ref{T-injective} holds for $f$ if and only if it holds for $fu^{-1}$. Therefore we may replace $f$ by $fu^{-1}$, and assume that $f=M-g$ such that $M=\diag(M_1, \dots, M_n)$, $M_k>\sum_{m\in [n]}\|g^{(km)}\|_1$ for all $k\in [n]$, and $P\supseteq \bigcup_{k, m\in [n]}\supp(g^{(km)})$.

Put
$$Y=S_f^\Gamma, \bar{P}=P\cup \{e_\Gamma\},  \bar{M}=\max_{k\in [n]}M_k, A=\supp(g)\subseteq \Gamma\times [n]^2, L_k=\sum_{m\in [n]}\|g^{(km)}\|_1$$
for $k\in [n]$.
Let $\nu$ be a  probability measure on $S_f$ such that $(\psi_k)_*\nu$ is the uniform probability measure on $\{0, 1, \dots, M_k-1\}$ for every $k\in [n]$ where $\psi_k: S_f\rightarrow \{0, 1, \dots, M_k-1\}$ is the projection. For any subset $E$ of $\Gamma$, denote by $\nu^E$ the product measure on
$S_f^E$ with base measure $\nu$, and by $\pi_E$ the restriction map
$S_f^\Gamma\rightarrow S_f^E$.
Note that for any closed subset $Y'$ of $Y$ and any $t, s\in \Gamma$, we have
\begin{align} \label{E-same}
\nu^{ts\bar{P}}(\pi_{ts\bar{P}}tY')=\nu^{s\bar{P}}(\pi_{s\bar{P}}Y').
\end{align}

Fix an integer $N\ge \bar{M}\|(f^*)^{-1}\|_{1, \infty}$.
Denote by $V$ the set of nonzero elements in  $W=(\{-N, -N+1, \dots, N\}^\Gamma)^{[n]}\subseteq (\Zb^n)^\Gamma\cap (\ell^\infty_\Rb(\Gamma))^n$, and set
$$ Z=\{(y, c)\in Y\times V: y+cf^*\in Y\}.$$
Denote by $\varphi$ the projection $Y\times V\rightarrow Y$. For $z=(z_1, \dots, z_n)\in (\ell^\infty_\Rb(\Gamma))^n$, put
$$\|z\|_\infty=\max_{k\in [n]}\|z_k\|_\infty.$$

\begin{lemma} \label{L-bad set}
We have
$$\{y\in Y: |\phi_f^{-1}(\phi_f(y))\cap Y|>1\}=\varphi(Z).$$
\end{lemma}
\begin{proof} Let $y, y'\in Y$ be distinct such that $\phi_f(y)=\phi_f(y')$.  Then $c:=y'(f^*)^{-1}-y(f^*)^{-1}$ is nonzero and is in $(\Zb^n)^\Gamma$.
Note that
$$ \|c\|_\infty\le \|y'-y\|_\infty\|(f^*)^{-1}\|_{1, \infty}\le \bar{M}\|(f^*)^{-1}\|_{1, \infty}\le N.$$
Thus $c\in W$, whence $c\in V$.  Since $y'=y+cf^*\in Y$, we have $(y, c)\in Z$, thus $\{y\in Y: |\phi_f^{-1}(\phi_f(y))\cap Y|>1\}\subseteq \varphi(Z)$.

Conversely, let $y\in \varphi(Z)$. Say, $(y, c)\in Z$. Then $y+cf^*\in Y$ is not equal to $y$, and $\phi_f(y)=\phi_f(y+cf^*)$. Therefore $\varphi(Z)\subseteq \{y\in Y: |\phi_f^{-1}(\phi_f(y))\cap Y|>1\}$ as desired.
\end{proof}

Note that $V$ and $Z$ are $F_\sigma$-subsets of $W$ and $Y\times W$ respectively. Thus $\varphi(Z)$ is an $F_\sigma$-subset of $Y$.
For each $(j, s, k) \in [N]\times \Gamma\times [n]$, put
$$ Z^+_{j, s, k}=\{(y, c)\in Z: \|c\|_\infty=c_{s, k}=j\} \mbox{ and } Z^-_{j, s, k}=\{(y, c)\in Z: \|c\|_\infty=-c_{s, k}=j\}.$$
Then $\varphi(Z^\dag_{j, s, k})$ is a closed subset of $Y$ for each $(j, s, k)\in [N]\times \Gamma\times [n]$ and $\dag\in \{+, -\}$, and
\begin{align} \label{E-union}
Z=\bigcup_{(j, s, k)\in [N]\times \Gamma\times [n], \dag\in \{+, -\}}Z^\dag_{j, s, k}.
\end{align}

For any $(j, k)\in [N]\times [n]$,  $\dag\in \{+, -\}$,  and $s, t\in \Gamma$, note that $(ty, tc)\in Z^\dag_{j, ts, k}$ for every $(y, c)\in Z^\dag_{j, s, k}$. It follows that $\varphi(Z^\dag_{j, ts, k})=t\varphi(Z^\dag_{j, s, k})$, whence $\nu^{ts\bar{P}}(\pi_{ts\bar{P}}\varphi(Z^\dag_{j, ts, k}))=\nu^{s\bar{P}}(\pi_{s\bar{P}}\varphi(Z^\dag_{j, s, k}))$ by \eqref{E-same}. For each $j\in [N]$, put
$$p_j:=\max_{k\in [n], \dag\in \{+, -\}}\nu^{s\bar{P}}(\pi_{s\bar{P}}\varphi(Z^\dag_{j, s, k}))\ge 0$$
for all $s\in \Gamma$.

The key fact for the proof of Theorem~\ref{T-injective} is the following lemma.

\begin{lemma} \label{L-bound2}
For each $j\in [N]$ we have $p_j\le \frac{\bar{M}-1}{\bar{M}}p_j$.
\end{lemma}

Let us show first how to derive Theorem~\ref{T-injective} from Lemma~\ref{L-bound2}.

\begin{proof}[Proof of Theorem~\ref{T-injective}]
From Lemma~\ref{L-bound2} we obtain $p_j=0$ for all $j\in [N]$. For any $(j, s, k)\in [N]\times \Gamma\times [n]$ and $\dagger\in \{+, -\}$, since
$\nu^\Gamma(\varphi(Z^{\dagger}_{j, s, k}))\le p_j$, we conclude that $\nu^\Gamma(\varphi(Z^\dag_{j, s, k}))=0$. From \eqref{E-union} we get $\nu^\Gamma(\varphi(Z))=0$. Then Theorem~\ref{T-injective} follows from Lemma~\ref{L-bad set}.
\end{proof}

The rest of this section is devoted to the proof of Lemma~\ref{L-bound2}.

Fix $(j, s, k)\in [N]\times \Gamma\times [n]$. Put
$$A_k=\{(a, m): (a, k, m)\in A\}.$$
Then $L_k=\sum_{(a, m)\in A_k}|g^{(km)}_a|$.

Let $B_k\subseteq A_k$. Put $L_{B_k}=\sum_{(a, m)\in B_k}|g^{(km)}_a|$. For $B_k=\emptyset$, we set $L_{B_k}=0$. For $\dag\in \{+, -\}$, put
\begin{align*}
Z^\dag_{j, s, k, B_k}&=\{(y, c)\in Z^\dag_{j, s, k}: c_{sa, m}=\dag\sgn(g^{(km)}_a)j \mbox{ for all } (a, m)\in B_k, \\
& \quad \quad \quad \quad \quad c_{sa, m}\neq \dag\sgn(g^{(km)}_a)j \mbox{ for all } (a, m)\in A_k\setminus B_k\},
\end{align*}
and
$$Z^\dag_{j, s, k, B_k, i}=\{(y, c)\in Z^\dag_{j, s, k, B_k}: y_{s, k}=i\}$$
for each $0\le i\le M_k-1$. Then for each $\dag\in \{+, -\}$ we have
\begin{align} \label{E-disjoint union}
 Z_{j, s, k}^\dag=\bigsqcup_{0\le i\le M_k-1}\bigsqcup_{B_k\subseteq A_k}Z_{j, s, k, B_k, i}^\dag.
 \end{align}

\begin{lemma} \label{L-range of i}
Let $0\le i\le M_k-1$ and $B_k\subseteq A_k$. The following hold:
\begin{enumerate}
\item If $Z^+_{j, s, k, B_k, i}$ is nonempty, then $0\le i\le L_{B_k}-1$.
\item If $Z^-_{j, s, k, B_k, i}$ is nonempty, then $M_k-L_{B_k}\le i\le M_k-1$.
\end{enumerate}
\end{lemma}
\begin{proof}
For any $(y, c)\in Z$ we have
\begin{align} \label{E-range}
 (cf^*)_{s, k}&=c_{s, k}M_k-\sum_{(a, m)\in A_k}c_{sa, m}(g^*)^{(mk)}_{a^{-1}}\\
\nonumber &=c_{s, k}M_k-\sum_{(a, m)\in A_k}c_{sa, m}g^{(km)}_a \\
 \nonumber &= c_{s, k}M_k-\sum_{(a, m)\in B_k}c_{sa, m}g^{(km)}_a-\sum_{(a, m)\in A_k\setminus B_k}c_{sa, m}g^{(km)}_a.
  \end{align}

(1). Let $(y, c)\in Z^+_{j, s, k, B_k}$.  We have
\begin{align*}
 (cf^*)_{s, k} &\overset{\eqref{E-range}}=jM_k-\sum_{(a, m)\in B_k}j|g^{(km)}_a|-\sum_{(a, m)\in A_k\setminus B_k}c_{sa, m}g^{(km)}_a\\
 &\ge jM_k-\sum_{(a, m)\in B_k}j|g^{(km)}_a|-\sum_{(a, m)\in A_k\setminus B_k}(j-1)|g^{(km)}_a|\\
 &=j(M_k-L_k)+L_{A_k\setminus B_k}\\
 &\ge M_k-L_k+L_{A_k\setminus B_k}=M_k-L_{B_k}.
 \end{align*}
 Since $y+cf^*\in Y$, we have $y_{s, k}+(cf^*)_{s, k}\le M_k-1$, whence
 $$0\le y_{s, k}\le L_{B_k}-1.$$
 Therefore $Z^+_{j, s, k, B_k, i}=\emptyset$ unless $0\le i\le L_{B_k}-1$.

(2). Let  $(y, c)\in Z^-_{j, s, k, B_k}$. We have
\begin{align*}
 (cf^*)_{s, k} &\overset{\eqref{E-range}}=-jM_k+\sum_{(a, m)\in B_k}j|g^{(km)}_a|-\sum_{(a, m)\in A_k\setminus B_k}c_{sa, m}g^{(km)}_a\\
 &\le -jM_k+\sum_{(a, m)\in B_k}j|g^{(km)}_a|+\sum_{(a, m)\in A_k\setminus B_k}(j-1)|g^{(km)}_a|\\
 &=-j(M_k-L_k)-L_{A_k\setminus B_k}\\
 &\le -(M_k-L_k)-L_{A_k\setminus B_k}=-M_k+L_{B_k}.
 \end{align*}
 Since $y+cf^*\in Y$, we have $y_{s, k}+(cf^*)_{s, k}\ge 0$, whence
 $$M_k-L_{B_k}\le y_{s, k}\le M_k-1.$$
 Therefore $Z^-_{j, s, k, B_k, i}=\emptyset$ unless $M_k-L_{B_k}\le i\le M_k-1$.
 \end{proof}

 For $\dag\in \{+, -\}$ we introduce two $\Rb$-valued Borel functions on $S_f^{sP}$ as follows:
$$u^\dag_k:=\sum_{0\le i\le M_k-1}\chi_{\pi_{sP}\varphi(\bigcup_{B_k\subseteq A_k}Z^\dag_{j, s, k, B_k, i})}$$
and
$$h^\dag_k:=\sum_{(a, m)\in A_k}|g^{(km)}_a|\chi_{\pi_{sa\bar{P}}\varphi(Z^{\dag\sgn(g^{(km)}_a)}_{j, sa, m})\times S_f^{sP\setminus sa\bar{P}}},$$
where $\chi_U$ denotes the characteristic function of a set $U\subseteq S_f^{sP}$.

 \begin{lemma} \label{L-measure to integral}
 We have
 \begin{align} \label{E-measure to integral}
 \nu^{s\bar{P}}(\pi_{s\bar{P}}\varphi(Z^+_{j, s, k}))\le \frac{1}{M_k}\nu^{sP}(u^+_k),
 \end{align}
 and
 \begin{align} \label{E-measure to integral2}
 \nu^{s\bar{P}}(\pi_{s\bar{P}}\varphi(Z^-_{j, s, k}))\le \frac{1}{M_k}\nu^{sP}(u^-_k).
 \end{align}
\end{lemma}
\begin{proof} We prove \eqref{E-measure to integral} first. Since $(\psi_k)_*\nu$ is the uniform probability measure on $\{0, \dots, M_k-1\}$, we have $(\psi_k)_*\nu(\{i\})=\frac{1}{M_k}$ for every $0\le i\le M_k-1$. Thus
 \begin{align*}
 \nu^{s\bar{P}}(\pi_{s\bar{P}}\varphi(Z^+_{j, s, k}))&\overset{\eqref{E-disjoint union}}=\nu^{s\bar{P}}\big(\pi_{s\bar{P}}\varphi\big(\bigcup_{0\le i\le M_k-1}\bigcup_{B_k\subseteq A_k}Z^+_{j, s, k, B_k, i}\big)\big)\\
 &=\sum_{0\le i\le M_k-1}\nu^{s\bar{P}}\big(\pi_{s\bar{P}}\varphi\big(\bigcup_{B_k\subseteq A_k}Z^+_{j, s, k, B_k, i}\big)\big)\\
 &\le \sum_{0\le i\le M_k-1}((\psi_k)_*\nu(\{i\}))\cdot \nu^{sP}\big(\pi_{sP}\varphi\big(\bigcup_{B_k\subseteq A_k}Z^+_{j, s, k, B_k, i}\big)\big)\\
 &= \frac{1}{M_k}\sum_{0\le i\le M_k-1}\nu^{sP}\big(\pi_{sP}\varphi\big(\bigcup_{B_k\subseteq A_k}Z^+_{j, s, k, B_k, i}\big)\big)\\
 &=\frac{1}{M_k}\nu^{sP}(u^+_k).
 \end{align*}
 
 The inequality \eqref{E-measure to integral2} is proved in the same way replacing $+$ everywhere by $-$. 
\end{proof}

\begin{lemma} \label{L-bound1}
We have
\begin{align} \label{E-bound1}
u^+_k\le h^+_k,
\end{align}
and 
\begin{align} \label{E-bound112}
u^-_k\le h^-_k.
\end{align}
\end{lemma}
\begin{proof} We prove \eqref{E-bound1} first. 
For each $C_k\subseteq A_k$, denote by $Y^+_{s, C_k}$ the set of $y\in S_f^{sP}$ satisfying $y\in  \pi_{sa\bar{P}}\varphi(Z^{\sgn(g^{(km)}_a)}_{j, sa, m})\times S_f^{sP\setminus sa\bar{P}}$ for every $(a, m)\in C_k$ and $y\not\in  \pi_{sa\bar{P}}\varphi(Z^{\sgn(g^{(km)}_a)}_{j, sa, m})\times S_f^{sP\setminus sa\bar{P}}$ for every $(a, m)\in A_k\setminus C_k$. Then the family $\{Y^+_{s, C_k}: C_k\subseteq A_k\}$ is a finite Borel partition of $S_f^{sP}$.

We have
\begin{align} \label{E-bound12}
\sum_{C_k\subseteq A_k}L_{C_k}\chi_{Y^+_{s, C_k}}&=\sum_{C_k\subseteq A_k}\sum_{(a, m)\in C_k}|g^{(km)}_a|\chi_{Y^+_{s, C_k}}\\
\nonumber &=\sum_{(a, m)\in A_k}|g^{(km)}_a|\sum_{(a, m)\in C_k\subseteq A_k}\chi_{Y^+_{s, C_k}}\\
\nonumber &=\sum_{(a, m)\in A_k}|g^{(km)}_a|\chi_{\pi_{sa\bar{P}}\varphi(Z^{\sgn(g^{(km)}_a)}_{j, sa, m})\times S_f^{sP\setminus sa\bar{P}}}=h^+_k.
\end{align}
For any $0\le i\le M_k-1$ and $B_k\subseteq A_k$, note that $Z^+_{j, s, k, B_k, i}\subseteq \bigcap_{(a, m)\in B_k}Z^{\sgn(g^{(km)}_a)}_{j, sa, m}$, whence
\begin{align*}
\pi_{sP}\varphi(Z^+_{j, s, k, B_k, i})&\subseteq \bigcap_{(a, m)\in B_k}\pi_{sP}\varphi(Z^{\sgn(g^{(km)}_a)}_{j, sa, m})\\
&\subseteq \bigcap_{(a, m)\in B_k}\pi_{sa\bar{P}}\varphi(Z^{\sgn(g^{(km)}_a)}_{j, sa, m})\times S_f^{sP\setminus sa\bar{P}}\\
&=\bigcup_{B_k\subseteq C_k\subseteq A_k}Y^+_{s, C_k}.
\end{align*}

If $Z^+_{j, s, k, B_k, i}\neq \emptyset$ and $B_k\subseteq C_k\subseteq A_k$, then by Lemma~\ref{L-range of i} we have
$$ 0\le i\le L_{B_k}-1\le  L_{C_k}-1.$$
Therefore for each $0\le i\le M_k-1$ we have
$$\bigcup_{B_k\subseteq A_k}\pi_{sP}\varphi(Z^+_{j, s, k, B_k, i})\subseteq \bigcup_{C_k\subseteq A_k, i\le L_{C_k}-1}Y^+_{s, C_k},$$
and hence
$$\chi_{\pi_{sP}\varphi(\bigcup_{B_k\subseteq A_k}Z^+_{j, s, k, B_k, i})}=\chi_{\bigcup_{B_k\subseteq A_k}\pi_{sP}\varphi(Z^+_{j, s, k, B_k, i})}\le \sum_{C_k\subseteq A_k, i\le L_{C_k}-1}\chi_{Y^+_{s, C_k}}.$$
Now we have
\begin{align*}
u^+_k&=\sum_{0\le i\le M_k-1}\chi_{\pi_{sP}\varphi(\bigcup_{B_k\subseteq A_k}Z^+_{j, s, k, B_k, i})}\\
&\le \sum_{0\le i\le M_k-1}\sum_{C_k\subseteq A_k, i\le L_{C_k}-1}\chi_{Y^+_{s, C_k}}\\
&=\sum_{C_k\subseteq A_k}\sum_{0\le i\le L_{C_k}-1}\chi_{Y^+_{s, C_k}}\\
&=\sum_{C_k\subseteq A_k}L_{C_k}\chi_{Y^+_{s, C_k}}\overset{\eqref{E-bound12}}=h^+_k.
\end{align*}
This proves \eqref{E-bound1}. 

Next we prove \eqref{E-bound112}. 
For each $C_k\subseteq A_k$, denote by $Y^-_{s, C_k}$ the set of $y\in S_f^{sP}$ satisfying $y\in  \pi_{sa\bar{P}}\varphi(Z^{-\sgn(g^{(km)}_a)}_{j, sa, m})\times S_f^{sP\setminus sa\bar{P}}$ for every $(a, m)\in C_k$ and $y\not\in  \pi_{sa\bar{P}}\varphi(Z^{-\sgn(g^{(km)}_a)}_{j, sa, m})\times S_f^{sP\setminus sa\bar{P}}$ for every $(a, m)\in A_k\setminus C_k$. Then the family $\{Y^-_{s, C_k}: C_k\subseteq A_k\}$ is a finite Borel partition of $S_f^{sP}$.

We have
\begin{align} \label{E-bound122}
\sum_{C_k\subseteq A_k}L_{C_k}\chi_{Y^-_{s, C_k}}&=\sum_{C_k\subseteq A_k}\sum_{(a, m)\in C_k}|g^{(km)}_a|\chi_{Y^-_{s, C_k}}\\
\nonumber &=\sum_{(a, m)\in A_k}|g^{(km)}_a|\sum_{(a, m)\in C_k\subseteq A_k}\chi_{Y^-_{s, C_k}}\\
\nonumber &=\sum_{(a, m)\in A_k}|g^{(km)}_a|\chi_{\pi_{sa\bar{P}}\varphi(Z^{-\sgn(g^{(km)}_a)}_{j, sa, m})\times S_f^{sP\setminus sa\bar{P}}}=h^-_k.
\end{align}
For any $0\le i\le M_k-1$ and $B_k\subseteq A_k$, note that $Z^-_{j, s, k, B_k, i}\subseteq \bigcap_{(a, m)\in B_k}Z^{-\sgn(g^{(km)}_a)}_{j, sa, m}$, whence
\begin{align*}
\pi_{sP}\varphi(Z^-_{j, s, k, B_k, i})&\subseteq \bigcap_{(a, m)\in B_k}\pi_{sP}\varphi(Z^{-\sgn(g^{(km)}_a)}_{j, sa, m})\\
&\subseteq \bigcap_{(a, m)\in B_k}\pi_{sa\bar{P}}\varphi(Z^{-\sgn(g^{(km)}_a)}_{j, sa, m})\times S_f^{sP\setminus sa\bar{P}}\\
&=\bigcup_{B_k\subseteq C_k\subseteq A_k}Y^-_{s, C_k}.
\end{align*}

If $Z^-_{j, s, k, B_k, i}\neq \emptyset$ and $B_k\subseteq C_k\subseteq A_k$, then by Lemma~\ref{L-range of i} we have
$$ M_k-1\ge i\ge M_k-L_{B_k}\ge  M_k-L_{C_k}.$$
Therefore for each $0\le i\le M_k-1$ we have
$$\bigcup_{B_k\subseteq A_k}\pi_{sP}\varphi(Z^-_{j, s, k, B_k, i})\subseteq \bigcup_{C_k\subseteq A_k, i\ge M_k-L_{C_k}}Y^-_{s, C_k},$$
and hence
$$\chi_{\pi_{sP}\varphi(\bigcup_{B_k\subseteq A_k}Z^-_{j, s, k, B_k, i})}=\chi_{\bigcup_{B_k\subseteq A_k}\pi_{sP}\varphi(Z^-_{j, s, k, B_k, i})}\le \sum_{C_k\subseteq A_k, i\ge M_k-L_{C_k}}\chi_{Y^-_{s, C_k}}.$$
Now we have
\begin{align*}
u^-_k&=\sum_{0\le i\le M_k-1}\chi_{\pi_{sP}\varphi(\bigcup_{B_k\subseteq A_k}Z^-_{j, s, k, B_k, i})}\\
&\le \sum_{0\le i\le M_k-1}\sum_{C_k\subseteq A_k, i\ge M_k-L_{C_k}}\chi_{Y^-_{s, C_k}}\\
&=\sum_{C_k\subseteq A_k}\sum_{M_k-1\ge i\ge M_k-L_{C_k}}\chi_{Y^-_{s, C_k}}\\
&=\sum_{C_k\subseteq A_k}L_{C_k}\chi_{Y^-_{s, C_k}}\overset{\eqref{E-bound122}}=h^-_k.
\end{align*}
This finishes the proof of  \eqref{E-bound112}. 
\end{proof}

We are ready to prove Lemma~\ref{L-bound2}.

\begin{proof}[Proof of Lemma~\ref{L-bound2}]
For any $k\in [n]$, we have
\begin{align*}
\nu^{s\bar{P}}(\pi_{s\bar{P}}\varphi(Z^+_{j, s, k}))&\overset{\eqref{E-measure to integral}}\le \frac{1}{M_k}\nu^{sP}(u^+_k)\\
&\overset{\eqref{E-bound1}}\le \frac{1}{M_k}\nu^{sP}(h^+_k)\\
&= \frac{1}{M_k}\sum_{(a, m)\in A_k}|g^{(km)}_a|\nu^{sa\bar{P}}(\pi_{sa\bar{P}}\varphi(Z^{\sgn(g^{(km)}_a)}_{j, sa, m}))\\
&\le \frac{1}{M_k}\sum_{(a, m)\in A_k}|g^{(km)}_a|p_j\\
&=\frac{L_k}{M_k}p_j\le \frac{M_k-1}{M_k}p_j\le \frac{\bar{M}-1}{\bar{M}}p_j,
\end{align*}
and
\begin{align*}
\nu^{s\bar{P}}(\pi_{s\bar{P}}\varphi(Z^-_{j, s, k}))&\overset{\eqref{E-measure to integral2}}\le \frac{1}{M_k}\nu^{sP}(u^-_k)\\
&\overset{\eqref{E-bound112}}\le \frac{1}{M_k}\nu^{sP}(h^-_k)\\
&= \frac{1}{M_k}\sum_{(a, m)\in A_k}|g^{(km)}_a|\nu^{sa\bar{P}}(\pi_{sa\bar{P}}\varphi(Z^{-\sgn(g^{(km)}_a)}_{j, sa, m}))\\
&\le \frac{1}{M_k}\sum_{(a, m)\in A_k}|g^{(km)}_a|p_j\\
&=\frac{L_k}{M_k}p_j\le \frac{M_k-1}{M_k}p_j\le \frac{\bar{M}-1}{\bar{M}}p_j.
\end{align*}
Therefore
$$p_j=\max_{k\in [n], \dag\in \{+, -\}}\nu^{s\bar{P}}(\pi_{s\bar{P}}\varphi(Z^\dag_{j, s, k}))\le \frac{\bar{M}-1}{\bar{M}}p_j.$$
\end{proof}

\section{Haar measure} \label{S-Haar}

The following result is due to Hayes \cite[Corollary 5.2]{Hayes22}. Though Hayes only treated the case $n=1$, his argument there works for any $n$. For convenience of the reader, we give a proof here.

\begin{proposition} \label{P-Haar}
Let $f\in M_n(\Zb\Gamma)$ be either positively row lopsided or positively column lopsided. Let $\nu$ be the uniform probability measure on $S_f$. Then $(\phi_f)_*\nu^\Gamma=\mu_{X_f}$.
\end{proposition}

\begin{lemma} \label{L-Fourier}
Let $f\in M_n(\Zb\Gamma)$ be invertible in $M_n(\ell^1_\Rb(\Gamma))$. Let $M_1, \dots, M_n$ be positive integers. Let $\nu$ be a probability measure on $S=\prod_{k\in [n]}\{0, 1, \dots, M_k-1\}\subseteq \Zb^n$. Put  $\mu=(\phi_f)_*\nu^\Gamma$, as a measure on $((\Rb/\Zb)^\Gamma)^n$. For any $h\in (\Zb\Gamma)^n$, we have
$$ \widehat{\mu}(h)=\prod_{s\in \Gamma}\widehat{\nu}((hf^{-1})_s).$$
\end{lemma}
\begin{proof} For any row vector $z\in \Rb^n$, write $z^{\rm t}$ for the transpose column vector of $z$. For any $y\in S^\Gamma$, we have
\begin{align*}
\exp(-2\pi i\left<\phi_f(y), h\right>)&=\exp(-2\pi i(\phi_f(y)h^*)_{e_\Gamma})=\exp(-2\pi i (\pi(y(f^*)^{-1})h^*)_{e_\Gamma})\\
&=\exp(-2\pi i(y(f^*)^{-1}h^*)_{e_\Gamma})=\exp(-2\pi i(y(hf^{-1})^*)_{e_\Gamma})\\
&=\exp\big(-2\pi i\sum_{s\in \Gamma}y_s((hf^{-1})_s)^{\rm t}\big).
\end{align*}
Now
\begin{align*}
\widehat{\mu}(h)&=\int_{((\Rb/\Zb)^\Gamma)^n}\exp(-2\pi i\left<x, h\right>)\, d\mu(x)=\int_{S^\Gamma}\exp(-2\pi i\left<\phi_f(y), h\right>)\, d\nu^\Gamma(y)\\
&=\int_{S^\Gamma}\exp\big(-2\pi i\sum_{s\in \Gamma}y_s((hf^{-1})_s)^{\rm t}\big)\, d\nu^\Gamma(y)\\
&=\prod_{s\in \Gamma}\int_{S}\exp(-2\pi iy_s((hf^{-1})_s)^{\rm t}\big)\, d\nu(y_s)=\prod_{s\in \Gamma}\widehat{\nu}((hf^{-1})_s).
\end{align*}
\end{proof}

Let $f\in M_n(\Zb\Gamma)$ be either positively row lopsided or positively column lopsided. We use the notation in Definitions~\ref{D-lopsided} and \ref{D-positive lopsided}.
Using Lemma~\ref{L-reduction} and arguing as  in the paragraph after it, we
may assume that $f=M-g$ such that $M=\diag(M_1, \dots, M_n)$, 
$M_k>\sum_{m\in [n]}\|g^{(km)}\|_1$ (resp. $M_k>\sum_{m\in [n]}\|g^{(mk)}\|_1$) for all $k\in [n]$ when $f$ is positively row (resp. column) lopsided, and $P\supseteq \bigcup_{m, k\in [n]}\supp(g^{(mk)})$.

\begin{lemma} \label{L-coordinate}
Let $h\in (\Zb\Gamma)^n\setminus (\Zb\Gamma)^nf$. Then there are some $(s, k)\in \Gamma\times [n]$ and $1\le j\le M_k-1$ such that $(hf^{-1})_{s, k}-j/M_k\in \Zb$.
\end{lemma}
\begin{proof}
We consider first the case $f$ is positively row lopsided.
  Let $x$ be the unique element in $[-1/2, 1/2)^{\Gamma \times [n]}$ satisfying that
  $hf^{-1}-x\in \Zb^{\Gamma\times [n]}$. Since $hf^{-1}\in (\ell^1_\Rb(\Gamma))^n$, we have
  $|(hf^{-1})_{s, k}|<1/2$ for all except finitely many $(s, k)\in \Gamma\times [n]$.
  Whenever $|(hf^{-1})_{s, k}|<1/2$, we have $(hf^{-1})_{s, k}=x_{s, k}$.
  Thus $hf^{-1}-x\in (\Zb\Gamma)^n$. Put $y=hf^{-1}-x$. Then $yf=h-xf$, whence $xf\in (\Zb\Gamma)^n$. Moreover, $xf\neq 0$ as $h\notin (\Zb\Gamma)^nf$. Thus
we can find some $(s_0, k_0)\in \Gamma\times [n]$ such that $(xf)_{s_0, k_0}\neq 0$ and $s_0\not\in tP$ for all $(t, m)\in \Gamma\times [n]$ with $(xf)_{t, m}\neq 0$. From
\begin{align} \label{E-coordinate1}
f^{-1}=(M(I_n-M^{-1}g))^{-1}=\sum_{l=0}^\infty(M^{-1}g)^lM^{-1}=M^{-1}+\sum_{l=1}^\infty(M^{-1}g)^lM^{-1}
\end{align}
it is easily checked that $x_{s_0, k_0}=(xff^{-1})_{s_0, k_0}=\frac{1}{M_{k_0}}(xf)_{s_0, k_0}\neq 0$.
Since $xf\in (\Zb\Gamma)^n$ and $x\in[\frac{-1}{2}, \frac{1}{2})^{\Gamma \times [n]}$, we conclude that $M_{k_0}x_{s_0, k_0}$ is a nonzero integer with absolute value at most $M_{k_0}/2$. Then there is a unique integer $1\le j\le M_{k_0}-1$ with $M_{k_0}x_{s_0, k_0}-j\in M_{k_0}\Zb$. Since $(hf^{-1})_{s_0, k_0}-x_{s_0, k_0}\in \Zb$, it follows that
$(hf^{-1})_{s_0, k_0}-j/M_{k_0}\in \Zb$.

The proof for the case of positively column lopsided $f$ is similar, replacing \eqref{E-coordinate1} by
\begin{align*}
f^{-1}=((I_n-gM^{-1})M)^{-1}=M^{-1}\sum_{l=0}^\infty(gM^{-1})^l=M^{-1}+M^{-1}\sum_{l=1}^\infty(gM^{-1})^l.
\end{align*}
\end{proof}

\begin{proof}[Proof of Proposition~\ref{P-Haar}] For $h\in (\Zb\Gamma)^nf$, we have $(hf^{-1})_s\in \Zb^n$ and thus $\widehat{\nu}((hf^{-1})_s)=1$ for every $s\in \Gamma$, whence $\widehat{(\phi_f)_*\nu^\Gamma}(h)=1$ by Lemma~\ref{L-Fourier}.

Let $h\in (\Zb\Gamma)^n\setminus (\Zb\Gamma)^nf$. By Lemma~\ref{L-coordinate} there are some $(s, k)\in \Gamma\times [n]$ and $1\le j\le M_k-1$ such that $(hf^{-1})_{s, k}-j/M_k\in \Zb$. Since $\nu$ is the uniform probability measure on $S_f$, we have $\widehat{\nu}((hf^{-1})_s)=0$. By Lemma~\ref{L-Fourier} we conclude that $\widehat{(\phi_f)_*\nu^\Gamma}(h)=0$.

We have shown that $\widehat{(\phi_f)_*\nu^\Gamma}(h)=\widehat{\mu_{X_f}}(h)$ for all $h\in (\Zb\Gamma)^n$. Therefore $(\phi_f)_*\nu^\Gamma=\mu_{X_f}$.
\end{proof}



\end{document}